\documentclass[10pt]{amsart}
\usepackage{graphicx}
\usepackage{dcolumn}
\usepackage{bm}
\usepackage[utf8]{inputenc}
\usepackage[T1]{fontenc}
\usepackage{mathptmx}
\usepackage[a4paper,margin=3cm]{geometry}%
\usepackage[pagebackref]{hyperref}%
\usepackage{bbm,lmodern,mathtools,amssymb,amsthm,tikz,graphicx}%
\usepackage[utf8]{inputenc}%
\usepackage[T1]{fontenc}%
\usepackage{enumerate}

\newtheorem{theorem}{Theorem}[section]%
\newtheorem{lemma}[theorem]{Lemma}%
\newtheorem{remark}[theorem]{Remark}%

\begin{document}
	
	\title{Macroscopic and edge behavior of a planar jellium}
	
	\author{Djalil Chafaï}%
    \address[DC]{CEREMADE, Université Paris-Dauphine, PSL University, France.}
	\email{{djalil(at)chafai.net}}
	\urladdr{\url{http://djalil.chafai.net/}}
	
	\author{David García-Zelada}%
	\address[DGZ]{Institut de Mathématiques de Marseille, Aix-Marseille Université, France.}
	\email{{david.garcia-zelada(at)univ-amu.fr}}
	\urladdr{\url{https://davidgarciaz.wixsite.com/math}}
	
	\author{Paul Jung}
	\address[PJ]{KAIST, Daejeon, Korea.}
    \email{{pauljung(at)kaist.ac.kr}}
	\urladdr{\url{http://mathsci.kaist.ac.kr/~pauljung/}}
	
	\date{Winter 2020 compiled \today}
	
	\begin{abstract}
      We consider a planar Coulomb gas in which the external potential is
      generated by a smeared uniform background of opposite-sign charge on a
      disc. This model can be seen as a two-dimensional Wigner jellium, not
      necessarily charge-neutral, and with particles allowed to exist beyond
      the support of the smeared charge. The full space integrability
      condition requires low enough temperature or high enough total smeared
      charge. This condition does not allow at the same time, total
      charge-neutrality and determinantal structure. The model shares
      similarities with both the complex Ginibre ensemble and the
      Forrester--Krishnapur spherical ensemble of random matrix theory. In
      particular, for a certain regime of temperature and total charge, the
      equilibrium measure is uniform on a disc as in the Ginibre ensemble,
      while the modulus of the farthest particle has heavy-tailed fluctuations
      as in the Forrester--Krishnapur spherical ensemble. We also touch on a
      higher temperature regime producing a crossover equilibrium measure, as
      well as a transition to Gumbel edge fluctuations. More results in the
      same spirit on edge fluctuations are explored by the second author
      together with Raphael Butez.
	\end{abstract}
	\maketitle

\section{Introduction}
Coulomb gases are systems of charged particles, all of the same sign, where the pair potential between particles is of Coulomb type. If the space in which the particles live is not compact, then an external potential is required to
confine the particles from `going off to infinity', since they are all of the same sign.
Wigner jelliums are Coulomb gases for which this external potential is precisely the Coulomb
potential generated by a charged background of opposite sign. Typically, for the jellium, one imposes an additional constraint that all particles live in some compact region which is equivalent to having an infinite external potential on the complement of this region.
We study a simple planar jellium obtained using a uniform background on a
centered disc, but where the particles are not confined to live on this disc. Our analysis reveals that this model, seen as a Coulomb gas,
cannot be both charge-neutral and determinantal. Moreover, this
model shares similarities with both the Ginibre and the spherical model, see
Figure \ref{fi:conf}. In particular, it has a uniform equilibrium on the
disc and has heavy-tailed fluctuations at the edge. 

The rest of the introduction is devoted to, firstly, a more precise description of the Wigner jellium and its link to the Coulomb gas, and secondly,  the main results concerning the particular unconfined jellium model that we study. Section \ref{se:cg} is devoted to general notions and facts on planar potential
theory including planar Coulomb gases. It gathers the main known examples and
results needed to understand our main results. Section \ref{se:proofs} is
devoted to the proofs of our main results. Finally Section \ref{se:hist}
provides some historical remarks on the jellium and Coulomb gases.

\begin{figure}[htbp]
  \centering
  \begin{tikzpicture}
    \draw[->] (0,0) -- (5,0) node[right] {$|x|$};
    \draw[->] (0,-1) -- (0,2) node[above] {$V(x)=-\frac{\alpha}{n}U_\rho(x)$};
    \draw[domain=0:1,smooth,variable=\x,blue,very thick] plot
    ({\x},{(\x^2-1)/2});
    \draw[domain=1:5,smooth,variable=\x,blue,very thick] plot
    ({\x},{ln(\x)});
    \draw (0,0) node [left] {$0$};
    \draw (1,0) node [below] {$1$};
    \draw (0,-.5) node [left] {$-\frac{\alpha}{2n}$};
  \end{tikzpicture}
  \caption{Plot of the external potential $V=-\frac{\alpha}{n}U_\rho$ used later in Lemma
    \ref{le:conf} when the radius of the disc is one. Here, $n$ is the number of particles and $\alpha$ is total charge of the opposite-signed background. In the neighborhood of the origin, the behavior
    is quadratic just like the potential of the Ginibre ensemble, while
    outside this neighborhood, the behavior is logarithmic just like the
    potential of the spherical ensemble. \label{fi:conf}}
\end{figure}
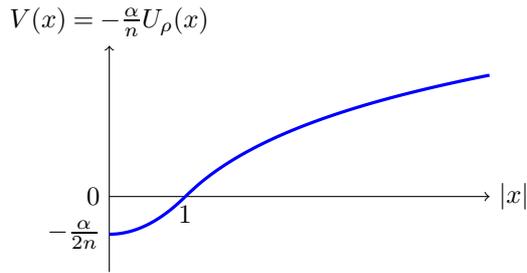

\subsection{Wigner jelliums as Coulomb gases}

Following Wigner \cite{TF9383400678}, let us consider $n$ unit negatively
charged particles (electrons) at positions $x_1,\ldots,x_n$ in
$\mathbb{C}=\mathbb{R}^2$, lying in a positive background of total charge
$\alpha>0$ smeared according to a probability measure $\rho$ on $\mathbb{C}$
with finite Coulomb energy $c=\mathcal{E}(\rho)$, see \eqref{eq:cE} for a
general definition of $\mathcal{E}(\cdot)$. We could alternatively suppose that the particles are
positively charged (ions) and the background is negatively charged
(electrons), this reversed choice would not affect the analysis of the model.
The total energy of the system, counting each pair a single time, is given by
\[
\sum_{i<j}g(x_i-x_j)%
-\alpha\sum_{i=1}^nU_\rho(x_i)%
+\alpha^2c,
\]
where $U_\rho=-(\log\left|\cdot\right|)*\rho$ is the logarithmic potential of
$\rho$. This matches the energy formula \eqref{eq:E} of a Coulomb gas with
$V=-\frac{\alpha}{n}U_\rho$. This observation leads us to define the
\emph{jellium model} on $S\subset\mathbb{C}$ with total background charge $\alpha>0$
and background distribution $\rho$ with $\mathrm{supp}\rho\subset S$ as being
the Coulomb gas \eqref{eq:P} on the full space $\mathbb{C}$, with potential
$V$ given by
\[
V=
\begin{cases}
-\frac{\alpha}{n}U_\rho
&\text{on $S$}\\
+\infty
&\text{on $S^c$}.
\end{cases}
\]
We say that the system is \emph{(charge) neutral} when $\alpha=n$. We say
that it is \emph{uniform} when $\rho$ is the uniform distribution on some
compact subset of $\mathbb{C}$. The great majority of jellium
models studied in the literature are charge-neutral and satisfy
$S=\mathrm{supp}\rho$.

Conversely, from the energy formula above and the inversion formula for the
logarithmic potential \eqref{eq:inv}, a Coulomb gas with sub-harmonic
potential $V$ (meaning $\Delta V\geq0$) could be seen as a jellium as above
with $\alpha\rho=\frac{\Delta V}{2\pi}\mathrm{d}\ell_{\mathbb{C}}$ on
$S=\mathbb{C}$ where $\ell_{\mathbb{C}}$ stands for the Lebesgue measure on
$\mathbb{C}$, but such a $\rho$ is not necessarily a probability measure. When
$V$ is not sub-harmonic then $\rho$ is additionally no longer a positive
measure but we can still interpret it as a background with opposite charge on
$\{\Delta V<0\}$. 

The famous example of the complex Ginibre ensemble is a Coulomb gas with
potential $V=\left|\cdot\right|^2$, for which $\Delta V$ is constant, leading
to an interpretation of this Coulomb gas as a degenerate jellium on the full
space $\mathbb{C}$ with Lebesgue background. The beautiful example of the
Forrester--Krishnapur spherical ensemble is a Coulomb gas with potential
$V=(1+1/n)\log(1+\left|\cdot\right|^2)$, for which
$\Delta V=4(1+1/n)/(1+\left|\cdot\right|^2)^2$, leading to an interpretation
of this Coulomb gas as a jellium on the full space with a heavy tailed
background. We can also consider such a background--potential inverse problem
for the one-dimensional log-gases of random matrix theory, which can be seen
as two-dimensional Coulomb gases confined to the real line, such as the
Gaussian Unitary Ensemble. For instance it follows from the discussion in
\cite[Section 1.4]{MR2641363} that the logarithmic potential of the background
measure of total charge $n$ with Lebesgue density
\[
x\mapsto\frac{\sqrt{2n}}{\pi}\sqrt{1-\frac{x^2}{2n}}\mathbf{1}_{|x|<\sqrt{2n}}
\]
is given on the interval $S=[-\sqrt{2n},\sqrt{2n}]$ by
\[
x\mapsto \frac{x^2}{2}+\frac{n}{2}\Bigr(\log\frac{n}{2}-1\Bigr).
\]

\subsection{Model and main results}

In this note, we focus on a very simple planar jellium on the full space
$S=\mathbb{C}$, seen as a Coulomb gas $P_n$ defined by \eqref{eq:P} with
\[
V=-\frac{\alpha}{n} U_\rho,\quad \alpha>0,
\]
where $\rho$ is the uniform probability distribution on the closed centered
disc
\[
D_R=\{z\in\mathbb{C}:|z|\leq R\}
\]
of radius $R>0$. We study the macroscopics and the edge asymptotics of this planar Coulomb
gas. We let $\alpha$ and $\beta$ depend on $n$ and we proceed in an asymptotic
analysis as $n\to\infty$. The potential $V$ depends on $n$. Our analysis
reveals that this special model shares similarities with the complex Ginibre
and the Forrester--Krishnapur spherical ensembles.

We look at (a) when this system is well-defined, and in the case it is
well-defined, (b) global asymptotics at the level of the equilibrium measure,
and (c) edge behavior in the sense of asymptotic analysis of the particle
farthest from the origin.

%
%

We first need requirements under which the Boltzmann--Gibbs measures exist.
The following lemma says that when the total charge of the background is high
enough, the confinement effect on the gas is strong enough to define the
Boltzmann--Gibbs measure. The condition is natural for Coulomb gases, see for
instance \cite{MR3215627}. Note that the condition does not allow, at the same
time, both charge-neutrality and determinantal structure.

We use the notation $Z_n$ and $P_n$ for the partition function and Gibbs measure of the system of $n$ particles, respectively (see \eqref{eq:Z} and \eqref{eq:P} below for a precise mathematical formulation).

\begin{lemma}[Confinement or integrability condition]\label{le:conf}
  We have
  \[
	Z_n<\infty%
	\quad\text{if and only if}\quad%
	\alpha-n>\frac{2}{\beta}-1.
  \]
  Moreover if this condition holds then $P_n$ is a Coulomb gas with an
  external potential
  \[
	V(x)
	=\frac{\alpha}{2n}\Bigr(\frac{|x|^2}{R^2}-1+2\log R\Bigr)\mathbf{1}_{|x|\leq
      R}+\frac{\alpha}{n}\log|x|\mathbf{1}_{|x|>R}.
  \]
  In particular, in the determinantal case $\beta=2$, the condition on
  $\alpha$ reads $\alpha>n$.\\ On the other hand, in the neutral case
  $\alpha=n$, the condition on $\beta$ reads $\beta>2$.
\end{lemma}

The proof of Lemma \ref{le:conf} is given in Section \ref{se:le:conf} and a plot of $V$ is provided in Figure \ref{fi:conf}.

Note that the condition does not depend on $R$.
Also note that the potential matches the one of the Ginibre ensemble when restricted to the
disc of radius $R$, while it is similar to the one of the spherical ensemble
when restricted to the region outside of the disc of radius $R$.

\medskip

We use the notation $X_n=(X_{n,1},\ldots,X_{n,n})$ and $\mu_{X_n}$ for the random vector of particle locations and the corresponding empirical distribution, respectively (see \eqref{eq:muXn} below).

\begin{theorem}[First order global asymptotics: low temperature regime]
	\label{th:low}
	Suppose that both $\alpha=\alpha_n$ and $\beta=\beta_n$ may depend on $n$ in
	such a way that
	\[
	\lim_{n\to \infty} n\beta_n = \infty
	\quad\text{and}\quad
	\lim_{n \to \infty} \frac{\alpha_n}{n} = \lambda\ge 1,
	\]
	and, if $\lambda=1$, that $\alpha_n-n>\frac{2}{\beta_n}-1$ for $n$ large
    enough. Then for $n$ large enough $Z_n<\infty$, and $P_n$ is well-defined.
    Moreover, regardless of the way we define the sequence of probability
    measures ${(P_n)}_n$ on the same probability space, we have that almost
    surely,
	\[
	\lim_{n \to \infty}\mathrm{d}_{\mathrm{BL}} (\mu_{X_n}, \mu_*) = 0,
	\]
	where $\mu_*$ is the uniform distribution on $D_{R/\sqrt \lambda}$.
\end{theorem}

The proof of Theorem \ref{th:low} is given in Section \ref{se:th:low}.

Note that the low temperature regime contains the determinantal case
$\beta=2$.
Also note that the case $\lambda<1$ is useless since Lemma \ref{le:conf} tells us in this
case that $Z_n=\infty$.

\begin{theorem}[First order global asymptotics: high temperature regime]
	\label{th:high}
	Suppose that both $\alpha=\alpha_n$ and $\beta=\beta_n$ may depend on $n$ in
	such a way that
	\[
	\lim_{n\to \infty} n\beta_n = \kappa>0
	\quad\text{and}\quad
	\lim_{n \to \infty}\frac{\alpha_n}{n} = \lambda
	\quad\text{with}\quad
	\kappa(\lambda - 1) > 2.
	\]
	Then for $n$ large enough, $Z_n<\infty$, and $P_n$ is well-defined.
    Moreover, regardless of the way we define the sequence of probability
    measures ${(P_n)}_n$ on the same probability space, we have, almost
    surely,
	\[
	\lim_{n \to \infty} 
	\mathrm{d}_{\mathrm{BL}} (\mu_{X_n}, \mu_*) = 0
	\]
	where $\mu_*$ has a density $\varphi$ that satisfies the following equation
	on its support
	\[
	\Delta\log\varphi%
	= 2\pi\kappa\Bigr(\varphi-\lambda\frac{\mathbf{1}_{|\cdot|\leq R}}{\pi R^2}\Bigr).
	\]
\end{theorem}

The proof of Theorem \ref{th:high} is given in Section \ref{se:th:high}.

\medskip

Our last results concern the fluctuation of the edge, in other words the
modulus of the farthest particle, in the determinantal case $\beta=2$. We
reveal a phase transition with respect to $\lambda$: the fluctuations are
heavy tailed if $\lambda=1$ and light tailed (Gumbel) if $\lambda>1$.

When $\lambda=1$, we know from Theorem \ref{th:low} that the equilibrium
$\mu_*$ is supported in $D_R$. The farthest particle will then ``feel'' $V$
outside $D_R$, which is, according to Lemma \ref{le:conf}, in this region,
logarithmic, and resembles that of the Forrester--Krishnapur spherical
ensemble. We can then expect that the fluctuations of the modulus of the
farthest particle will be then heavy tailed, however the fluctuation law may
differ from that of the spherical ensemble. This intuition is entirely
confirmed by the following theorem.

\begin{theorem}[Heavy-tailed edge]\label{th:edge}
	Suppose that $\beta=2$ and $\alpha=\alpha_n= n + \kappa_n$ with $\kappa_n>0$
	and $\lim_{n \to \infty}\kappa_n=\kappa > 0$, in such a way that in
	particular $\lim_{n\to\infty}\alpha_n/n=\lambda=1$. Then $Z_n<\infty$ by
	Lemma \ref{le:conf} and $P_n$ is well-defined. Moreover
	\[
	\max_{1\leq k\leq n}|X_{n,k}|
	\underset{n\to\infty}{\overset{\mathrm{law}}{\longrightarrow}}
	L
	\]
	where $L$ is the law with cumulative distribution function given by
	\[
        t\in\mathbb{R}\mapsto L((-\infty,t])
        =\prod_{k=0}^\infty
        \Bigr(1-\Bigr(\frac{R}{t}\Bigr)^{2(k+\kappa)}\Bigr)\mathbf{1}_{t\geq R}.
	\]
\end{theorem}

The proof of Theorem \ref{th:edge} is given in Section \ref{th:edge}.

Note that the law $L$ in Theorem \ref{th:edge} has a heavy (right) tail. 

Beyond the edge fluctuation, and following \cite[proof of Theorem
2.4]{raphael-david}, it is actually possible to show that the whole
(determinantal) point process converges as $n\to\infty$ to a Bergman point
process with explicit kernel (not related to the \textrm{erfc} special function as in
\cite{hedenmalm-wennman}). However the proof that we give of Theorem
\ref{th:edge} follows a simpler scheme based on
\cite{MR3215627, MR1148410}.

Note that in Theorem \ref{th:edge}, $L$ is supported in $[R,+\infty)$, and the
asymptotic fluctuations at the edge are thus one-sided. In some sense the
background produces here a hard edge.

\medskip

When $\lambda>1$, we know from Theorem \ref{th:low} that the equilibrium
$\mu_*$ is supported in $D_{R/\sqrt{\lambda}}$, which is included in $D_R$.
This suggests that if the farthest particle sticks to the edge of the limiting
support, it will ``feel'' $V$ inside $D_R$, which is, according to Lemma
\ref{le:conf}, quadratic and similar to the potential of a complex Ginibre
ensemble. We can then expect that the fluctuations of the modulus of the
farthest particle will be then light tailed and Gumbel distributed as for the
complex Ginibre ensemble. This is confirmed by our last theorem.

\begin{theorem}[Gumbel edge]\label{th:gumbel}
	Suppose that $\beta=2$ and that $\alpha=\alpha_n$ with
	$\lim_{n\to\infty}\alpha_n/n=\lambda>1$. Then $Z_n<\infty$ by Lemma
	\ref{le:conf} and $P_n$ is well-defined by \eqref{eq:P}. Moreover, if we
	define
	\[
	a_n = \frac{\sqrt {n c_n}}{C_n}\quad\text{and}\quad
	b_n = C_n\Bigr(1+\frac{1}{2}\sqrt {\frac{c_n}{n}}\Bigr)
	\]
	where
	$c_n = \log(n) - 2\log \log (n) - \log(2\pi)$ and
	$C_n=\sqrt{\frac{n}{\alpha_n}}R$,
	then
	\[
	\max_{1\leq k\leq n}|X_{n,k}|
	\underset{n\to\infty}{\overset{\mathbb{P}}{\longrightarrow}}
	\frac{R}{\sqrt{\lambda}}
    \quad\text{and}\quad
    a_n(\max_{1\leq k\leq n}|X_{n,k}|-b_n)%
	\overset{\mathrm{law}}{\underset{n\to\infty}{\longrightarrow}}
	G
	\]
	where $G$ is the Gumbel law with cumulative distribution function
	\[
      t\in\mathbb{R}\mapsto
      G((-\infty,t])=\mathrm{e}^{-\mathrm{e}^{-t}}.
	\]
\end{theorem}

The proof of Theorem \ref{th:gumbel} is given in Section \ref{se:th:gumbel}.

It is worth mentioning that following \cite{david-edge}, one can pass from the
heavy tailed law $L$ of Theorem \ref{th:edge} to the Gumbel light tailed law
$G$ of Theorem \ref{th:gumbel}. Namely, if for each $\kappa>0$, $\xi_{\kappa}$
is a random variable taking values in $[1,+\infty)$ with cumulative
distribution function
\[
t\in[1,+\infty)\mapsto
\mathbb P(\xi_{\kappa} \leq t)
= \prod_{k=0}^\infty (1 - t^{-2(k+\kappa)}),
\]
and if $\varepsilon_{\kappa}>0$ is the unique solution of
$\varepsilon_{\kappa}\exp(\kappa \varepsilon_{\kappa})=1$, then
\[
2\kappa\Bigr(\xi_\kappa - 1 -\frac{\varepsilon_{\kappa}}{2}\Bigr) 
\underset{\kappa\to\infty}{\overset{\mathrm{law}}{\longrightarrow}}
\mathrm{Gumbel}.
\]
Furthermore, following \cite{david-edge}, taking $\kappa=+\infty$ in Theorem
\ref{th:edge} leads to Gumbel fluctuations!

A simulation study can be done using the algorithm in \cite{MR3911782}, see
for instance Figure \ref{fi:low}.

\medskip

Note that the edge of one-dimensional models is considered in
\cite{Dhar2018,MR3817495,PhysRevLett.119.060601}.

\begin{figure}[htbp]
	\centering
	\includegraphics[width=.7\textwidth]{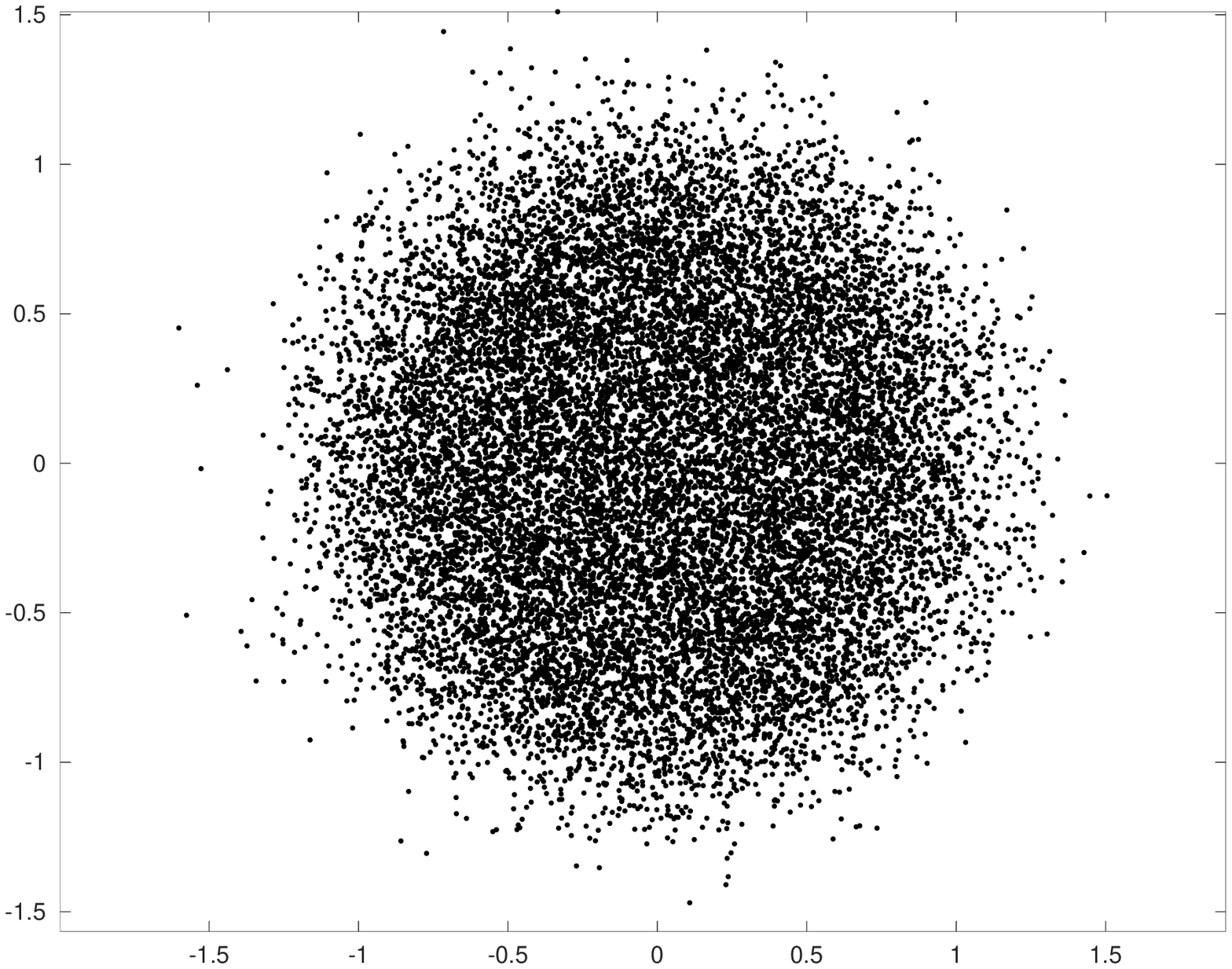}
	\includegraphics[width=.7\textwidth]{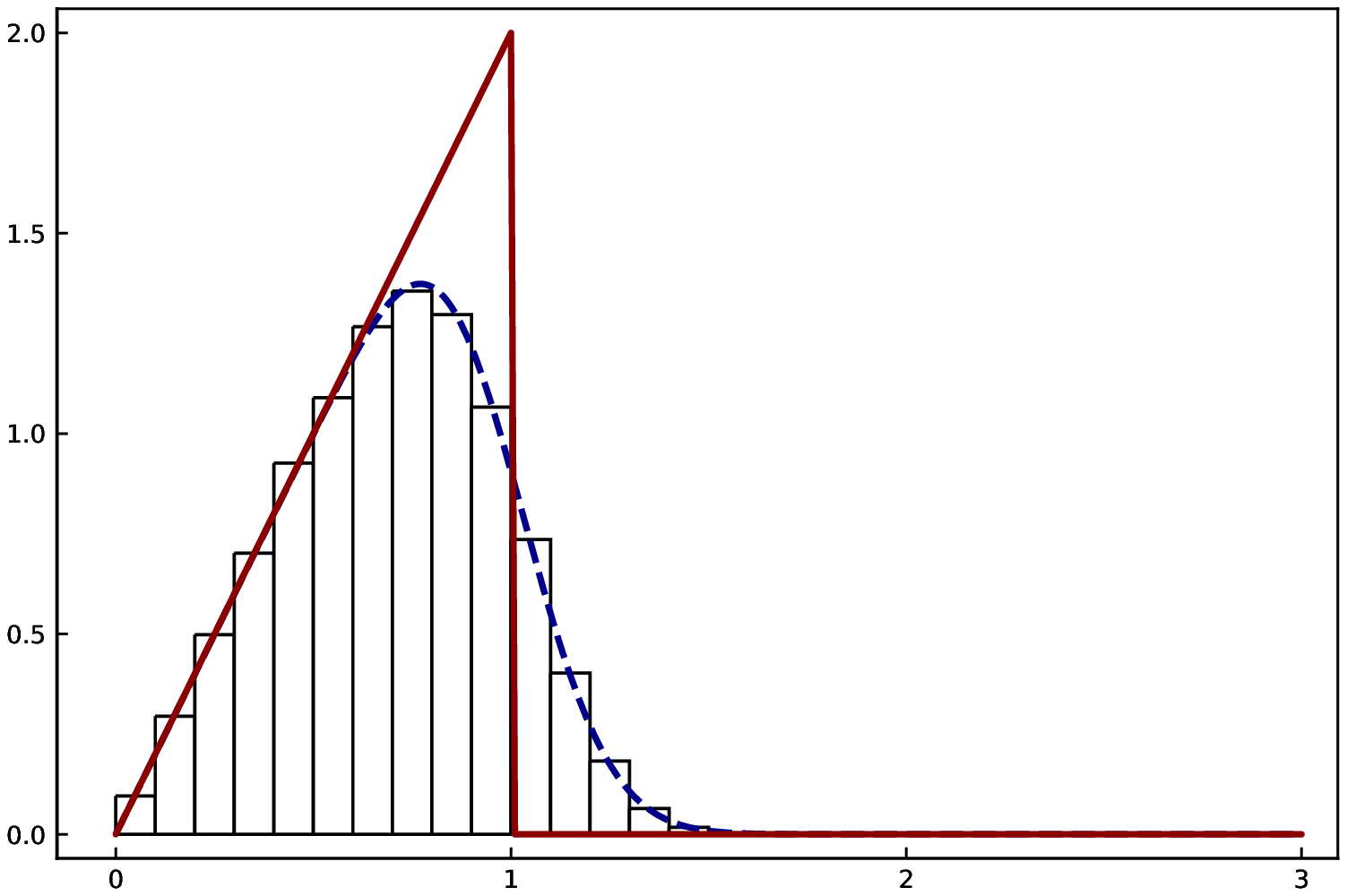}
	\caption{\label{fi:low}The top graphic is a plot of a simulation of
      $X_n\sim P_n$, $n=8$, illustrating Theorem \ref{th:low} and Theorem
      \ref{th:gumbel} in the case $R=2$ and $\lambda=4$. We used the algorithm
      from \cite{MR3911782} with \texttt{dt=.5} and \texttt{T=10e6}. About
      $10$ independent copies were simulated and merged and we retained only
      the last 10\% of the trajectories. The bottom graphic shows a histogram
      of the radii of the same data together with the non asymptotic radial
      density for the complex Ginibre ensemble (dashed line, exact formula
      from determinantal structure) and radial density of equilibrium measure
      (solid line).}
\end{figure}

Let us end this introduction with two open problems. We believe there should be a version of Theorem \ref{th:high} in the critical case $\kappa(\lambda - 1) = 2$, but
	it is unclear for us that the functional used in our proof of Theorem
	\ref{th:high} is well-defined in this case. For criticality, the proof of the
	large deviation principle may require a special proof. We should have
	$-\int \Delta \log\varphi\mathrm{d}\ell_{\mathbb C}=4\pi$. With the
	Gauss--Bonnet formula in mind, seeing $-\Delta \log\varphi$ as a curvature
	suggests a space of Euler characteristic one, which could be thought as the
	unit disc. A second open problem is to investigate the density $\varphi$ of Theorem \ref{th:high} as $\kappa\to\infty$, in other words, as the high temperature regime approaches the low temperature regime. The question here is to elucidate how or in what sense the density approaches the uniform distribution on $D_{R/\sqrt \lambda}$.

\section{Essential facts on planar potential theory}\label{se:cg}

This section gathers useful elements of two-dimensional potential theory. We
refer, for instance, to \cite{MR3308615,MR0350027,MR1817225,MR2908617} for
more details on the basic aspects of potential theory used in this note. The
Coulomb kernel $g$ in dimension $2$ is given on
$x\in\mathbb{C}=\mathbb{R}^2\setminus\{0\}$, by
\[
  g= -\log\left|\cdot\right|.
\]
It belongs to $\mathrm{L}^1_{\mathrm{loc}}(\ell_{\mathbb{C}})$ and constitutes
the fundamental solution of the Laplace or Poisson equation, namely
$\Delta g=-2\pi\delta_0$ in the sense of Schwartz distributions on
$\mathbb{R}^2$. In particular $g$ is super-harmonic, and harmonic on
$\mathbb{R}\setminus\{0\}$. The Coulomb potential at point $x\in\mathbb{C}$
generated by a distribution of charges (say electrons) modeled by a
probability measure $\mu$ on $\mathbb{C}$ such that
$g\mathbf{1}_{K^c}\in\mathrm{L}^1(\mu)$ for some large enough compact set $K$
is defined by
\begin{equation}\label{eq:U}
U_\mu(x)=(g*\mu)(x)=\int g(x-y)\mathrm{d}\mu(y) \in(-\infty,+\infty].
\end{equation}
We have $U_\mu\in\mathrm{L}^1_{\mathrm{loc}}(\ell_{\mathbb{C}})$ and the identity
$\Delta g=-2\pi\delta_0$ gives the inversion
formula
\begin{equation}\label{eq:inv}
\Delta U_\mu 
=-{2\pi}\mu.
\end{equation}
In particular $-U_\mu$ is sub-harmonic in the sense that $\Delta U_\mu\geq0$.
The Coulomb (self-interaction) energy of the (distribution of charges) $\mu$
is defined when it makes sense by
\begin{equation}\label{eq:cE}
\mathcal{E}(\mu)
=\frac{1}{2}\iint g(x-y)\mathrm{d}\mu(x)\mathrm{d}\mu(y)
=\frac{1}{2}\int U_\mu(x)\mathrm{d}\mu(x).
\end{equation}
A subset of $\mathbb{C}$ has finite capacity when its carries a probability
measure with finite Coulomb energy. When this is not the case, we say that the
set has zero capacity.

Let us denote by $\mathcal{P}(\mathbb{C})$ the set of probability measures on
$\mathbb{C}$, equipped with the topology of weak convergence with respect to
continuous and bounded test functions, and its associated Borel
$\sigma$-field. This topology is metrized by the bounded-Lipschitz metric
\[
\mathrm{d}_{\mathrm{BL}}(\mu,\nu)%
=\sup\Bigr\{
\int f\mathrm{d}(\mu-\nu):\|f\|_\infty\leq1,\|f\|_{\mathrm{Lip}}\leq1
\Bigr\}
\]
where $f:\mathbb{C}\to\mathbb{R}$ is measurable, $\|f\|_\infty=\sup_x|f(x)|$,
$\|f\|_{\mathrm{Lip}}=\sup_{x\neq y}\frac{|f(x)-f(y)|}{|x-y|}$. 

Let $V:\mathbb{C}\to\mathbb{R}\cup\{+\infty\}$ be a lower semi-continuous
function playing the role of an external potential, producing an external
electric field $-\nabla V$. If $V$ grows faster than $g$ at infinity, the
Coulomb energy $\mathcal{E}_V$ with external field is defined by
\[
\mu\in\mathcal{P}(\mathbb{C})
\mapsto
\mathcal{E}_V(\mu)=\mathcal{E}(\mu)+\int V\mathrm{d}\mu.
\]
It is lower semi-continuous with compact level sets, strictly convex, and it
admits a unique minimizer called the \emph{equilibrium measure} or Frostman
measure denoted
\[
\mu_*=\arg\min_{\mathcal{P}(\mathbb{C})}\mathcal{E}_V.
\]
From this variational formula, there exists a constant
$c$ 
such that except on a set of zero capacity, 
\begin{equation}\label{eq:UV}
  \begin{cases}
    U_{\mu_*}+V =c & \text{on the support of $\mu_*$},\\
    U_{\mu_*}+V \geq c&\text{outside}.
  \end{cases}
\end{equation}
In particular $V$ is sub-harmonic on the support of $\mu_*$.
Combined with \eqref{eq:inv}, we get, when $V$ has Lipschitz weak first
derivative, that
\begin{equation}\label{eq:mustar}
\mathrm{d}\mu_*=\frac{\Delta V}{2\pi}\mathrm{d}\ell_{\mathbb{C}}\quad\text{on the support of $\mu_*$}.
\end{equation}
In particular the probability measure $\mu_*$ is supported
in $\{\Delta V\geq0\}$. Furthermore $\mu_*$ is compactly supported when $V$
goes to $+\infty$ at $\infty$ fast enough.

\subsection{Planar Coulomb gases}

A planar Coulomb gas with $n$ particles, potential $V$, and inverse temperature
$\beta\geq0$ is the exchangeable Boltzmann--Gibbs probability measure $P_n$ on
$\mathbb{C}^n$ given by
\begin{equation}\label{eq:P}
\mathrm{d}P_n(x_1,\ldots,x_n)
=\frac{\mathrm{e}^{-\beta E_n(x_1,\ldots,x_n)}}{Z_n}
\mathrm{d}\ell_{\mathbb{C}}(x_1)\cdots\mathrm{d}\ell_{\mathbb{C}}(x_n)
\end{equation}
where
\begin{equation}\label{eq:E}
E_n(x_1,\ldots,x_n)=\sum_{i<j}g(x_i-x_j)+n\sum_{i=1}^nV(x_i)
\end{equation}
and
\begin{equation}\label{eq:Z}
Z_n=\int\mathrm{e}^{-\beta E_n(x_1,\ldots,x_n)}
\mathrm{d}\ell_{\mathbb{C}}(x_1)\cdots\mathrm{d}\ell_{\mathbb{C}}(x_n).
\end{equation}
The Coulomb gas is well defined when $Z_n<\infty$. It models a gas of unit
charged particles, or more precisely a random configuration of
unit charged particles. We should keep in mind that we play here with
electrostatics rather than with electrodynamics (no magnetic field). For all
$n$, we define the random empirical measure
\begin{equation}\label{eq:muXn}
  \mu_{X_n}=\frac{1}{n}\sum_{k=1}^n\delta_{X_{n,k}}
  \quad\text{where}\quad
  X_n=(X_{n,1},\ldots,X_{n,n})\sim P_n.
\end{equation}
The notation $\sim$ means that the random variable $X_n$ has law $P_n$. We
have $\mathbb{P}(\mu_{X_n}\in A)=P_n(\frac{1}{n}\sum_{k=1}^n\delta_{x_k}\in A)$
for any Borel subset $A\subset\mathcal{P}(\mathbb{C})$. In the low temperature
regime $\beta=\beta_n$ with $\lim_{n\to\infty}n\beta_n=\infty$, ${(\mu_{X_n})}_n$
satisfies the following \emph{large deviation principle}: for any Borel subset
$A\subset\mathcal{P}(\mathbb{C})$ with interior $\mathrm{int}(A)$ and closure
$\mathrm{clo}(A)$,
\begin{multline}\label{eq:LDP}
-\inf_{\mathrm{int}(A)}\mathcal{E}_V+\mathcal{E}_V(\mu_*)
\leq\varliminf_{n\to\infty}\frac{\log\mathbb{P}(\mu_{X_n}\in A)}{n^2\beta_n}\\
\leq\varlimsup_{n\to\infty}\frac{\log\mathbb{P}(\mu_{X_n}\in A)}{n^2\beta_n}
\leq-\inf_{\mathrm{clo}(A)}\mathcal{E}_V+\mathcal{E}_V(\mu_*).
\end{multline}
We refer to \cite{MR1746976,MR1606719,MR2760897,MR2926763,MR3262506,david-ihp}
for these large deviation principles for Coulomb gases.

For any Borel measures $\mu$ and $\nu$, we define the Kullback--Leibler
divergence or relative entropy of $\nu$ with respect to $\mu$ by
$D(\nu\mid\mu) =\int\log\frac{\mathrm{d}\nu}{\mathrm{d}\mu}\mathrm{d}\nu$ if
$\nu$ is absolutely continuous with respect to $\mu$, and
$D(\nu\mid\mu)=+\infty$ otherwise.

In the high temperature regime $\beta=\beta_n$ with
$\lim_{n\to\infty}n\beta_n=\kappa\in(0,+\infty)$, the same holds true with
$\mathcal{E}_V$ formally replaced by
\[
  \mathcal{E}_V + \frac{1}{\kappa} D(\cdot\mid \ell_{\mathbb C})
\]
or, equivalently, formally replaced by
\[
  \mathcal{E}+\frac{1}{\kappa} D(\cdot\mid\nu_V)
\]
where $\nu_V$ has a density proportional to $\mathrm{e}^{-\kappa V}$. We also
have to replace $\mu_*$ by the minimizer of
$\mathcal{E}_V + \frac{1}{\kappa} D(\cdot\mid \ell_{\mathbb C})$. This is
known as the \emph{crossover regime}, which interpolates between $\nu_V$ and
the minimizer of $\mathcal {E}_V$. The classical Sanov theorem corresponds
formally to this regime when we turn off the pair interaction by taking $g=0$.
This regime is considered in particular in
\cite{MR1145596,MR1678526,MR2909692,MR3262506,dupuis2015large,david-ihp,akemann-byun}.

From \eqref{eq:LDP}, for all $\varepsilon>0$, we get by taking
$A=\{\mu\in\mathcal{P}(\mathbb{C}):\mathrm{d}_{\mathrm{BL}}(\mu_{X_n},\mu_*)>\varepsilon\}$
that
\[
\sum_n\mathbb{P}(\mathrm{d}_{\mathrm{BL}}(\mu_{X_n},\mu_*)\geq\varepsilon)<\infty,
\]
which is a summable convergence in probability. By the Borel--Cantelli lemma, we obtain something  known in the probabilistic
literature as \emph{complete convergence}, see for instance \cite{MR1632875}.
In particular, regardless of the way we define the random vectors $X_n$ on the
same probability space, we have that almost surely,
\[
\lim_{n\to\infty}\mathrm{d}_{\mathrm{BL}}(\mu_{X_n},\mu_*)=0.
\]

\subsection{Determinantal exact solvability, Ginibre and spherical models}

It is useful to rewrite the density of the Coulomb gas $P_n$ defined in
\eqref{eq:P}, provided that $Z_n<\infty$, as
\begin{equation}\label{eq:det}
\frac{\mathrm{e}^{-\beta n\sum_{i=1}^nV(x_i)}}{Z_n}
\prod_{i<j}|x_i-x_j|^\beta.
\end{equation}
This includes plenty of famous models from random matrix theory including the
following couple of models, and we refer to
\cite{MR0173726,MR2489167,MR2552864,MR2540391,MR2641363,MR3309890,MR3888701}
for more information:

\begin{itemize}
	\item \emph{Complex Ginibre ensemble.} This corresponds to taking
	\[
	\beta=2\quad\text{and}\quad V=\frac{1}{2}\left|\cdot\right|^2.
	\]
	The equilibrium measure is uniform on the unit disc, namely
	\[
	\mathrm{d}\mu_*=\frac{\mathbf{1}_{\left|\cdot\right|\leq1}}{\pi}\mathrm{d}\ell_{\mathbb{C}}
	\]
	in accordance with \eqref{eq:mustar}. This Coulomb gas describes the
	eigenvalues of a Gaussian random complex $n\times n$ matrix $A$ with density
	proportional to $\mathrm{e}^{-\mathrm{Trace}(AA^*)}$ where
	$A^*=\bar{A}^\top$ is the conjugate-transpose of $A$. Equivalently, the
	entries of $A$ are independent and identically distributed with independent
	real and imaginary parts having a Gaussian law of mean $0$ and variance
	$1/(2n)$. This gas also appears in various other places in the
	mathematical physics literature, for instance as the modulus of the wave
	function in Laughlin's model of the fractional quantum Hall effect
	\cite{Laughlin1987}, in the description of the vortices in the
	Ginzburg--Landau model of superconductivity \cite{MR3309890}, and in a model
	of rotating trapped fermions \cite{PhysRevA.99.021602}.
	\item \emph{Forrester--Krishnapur spherical ensemble.} This corresponds to taking
	\[
	\beta=2\quad\text{and}\quad
	V=\frac{n+1}{2n}\log(1+|x|^2).
	\]
	The equilibrium measure is heavy tailed and given by
	\[
	\mathrm{d}\mu_*
	=\frac{1}{\pi(1+\left|\cdot\right|^2)^2}\mathrm{d}\ell_{\mathbb{C}}.
	\]
	The name of this gas comes from the fact that it is the image by the
    stereographical projection of the Coulomb gas on the sphere, with constant
    potential, onto the complex plane. This Coulomb gas describes the
    eigenvalues of $AB^{-1}$ where $A$ and $B$ are two independent copies of
    complex Ginibre random matrices. We can loosely interpret $AB^{-1}$ as a
    sort of matrix analogue of the Cauchy distribution since when $A$ and $B$
    are $1\times1$ matrices, this is precisely a Cauchy distribution.
\end{itemize}

The case $\beta=2$ has a remarkable integrable structure,
called a \emph{determinantal structure}, which provides exact solvability, see
for instance \cite{MR2760897,MR1083764,MR3342661,MR3215627}. More precisely,
if $\beta=2$ then for all $1\leq k\leq n$, the $k$-th dimensional marginal
distribution of the exchangeable probability measure $P_n$, denoted $P_{n,k}$,
has density proportional to
\[
(x_1,\ldots,x_k)\in\mathbb{C}^k\mapsto\det[K_n(x_i,x_j)]_{1\leq i,j\leq k}
\]
where $K_n$ is an explicit kernel which depends on $n$ and $V$. Since
$\mathbb{E}\mu_{X_n}=P_{n,1}$, it follows in particular that the density of
$\mathbb{E}\mu_{X_n}$ is proportional to $x\in\mathbb{C}\mapsto K_n(x,x)$.
Following \cite{MR1148410,MR3215627}, if $\beta=2$ and if $V$ is radially
symmetric, say $V=Q(\left|\cdot\right|)$, then the point process of radii or
moduli (this should be interpreted as a random multi-set)
\[
\{|X_{n,1}|,\ldots,|X_{n,n}|\}
\]
has the same law as the point process $\{Y_{n,1},\ldots,Y_{n,n}\}$ where
$R_{n,1},\ldots,R_{n,n}$ are independent (and not identically distributed)
random variables with $R_k$ of density proportional to
\begin{equation}\label{eq:kostlan}
t\in[0,+\infty)\mapsto
t^{2k-1}\mathrm{e}^{-2nQ(t)},
\quad 1\leq k\leq n.
\end{equation}
Following \cite{MR1986426,MR3215627,MR3615091,david-edge}, this allows for the
asymptotic analysis of the modulus of the farthest particle of the Coulomb gas
as $n\to\infty$. In particular, one can analyze:
\begin{itemize}
	\item\emph{Complex Ginibre ensemble.} For this gas, the equilibrium measure
	has an edge and the modulus of the farthest particle tends to this edge, and
	the fluctuation is described by a Gumbel law. Namely, following
	\cite{MR1986426,MR3215627}, if we define
	\[
	a_n = 2\sqrt {n c_n}
	\quad\text{and}\quad
	b_n = 1 + \frac{1}{2}\sqrt{\frac{c_n}{n}}
	\]
	where
	$c_n = \log(n) - 2\log \log (n) - \log(2\pi)$ then
	\begin{equation}\label{eq:ginibre-edge}
	\max_{1\leq k\leq n}|X_{n,k}|
	\underset{n\to\infty}{\overset{\mathbb{P}}{\longrightarrow}}
	1
	\quad\text{and}\quad  
	a_n(\max_{1\leq k\leq n}|X_{n,k}|-b_n)
	\underset{n\to\infty}{\overset{\mathrm{law}}{\longrightarrow}}
	G
	\end{equation}
	where $G$ is the Gumbel law with cumulative probability distribution
	\begin{equation}\label{eq:gumbelaw}
      t\in\mathbb{R}\mapsto G((-\infty,t])%
      =\mathrm{e}^{-\mathrm{e}^{-t}}.
	\end{equation}
	\item\emph{Forrester--Krishnapur spherical ensemble.} For this gas, the
	modulus of the farthest particle tends to infinity and has a heavy tail.
	Namely, following \cite{MR3215627,MR3615091,MR3761607}.
	\begin{equation}\label{eq:spherical-edge}
	\max_{1\leq k\leq n}|X_{n,k}|
	\underset{n\to\infty}{\overset{\mathbb{P}}{\longrightarrow}}
	+\infty
	\quad\text{and}\quad  
	\frac{1}{\sqrt{n}}\max_{1\leq k\leq n}|X_{n,k}|
	\underset{n\to\infty}{\overset{\mathrm{law}}{\longrightarrow}}
	F
	\end{equation}
	where $F$ is the probability distribution with cumulative distribution
	function
	\begin{equation}\label{eq:sphericaledgelaw}
	t\in\mathbb{R}\mapsto
	F((-\infty,t])=
	\prod_{k=1}^\infty\mathrm{e}^{-t^{-2}}\sum_{j=0}^{k-1}\frac{t^{-2j}}{j!}\mathbf{1}_{t\geq0},
	\end{equation}
	moreover this law is heavy tailed in the sense that
	\[
	1-F((-\infty,t])=F((t,+\infty))\underset{t\to\infty}{\sim} t^{-2}.
	\]
\end{itemize} 

The notation $\overset{\mathrm{law}}{\to}$ and $\overset{\mathbb{P}}{\to}$
stand for convergence in law and in probability respectively.

In random matrix theory and statistical physics, it is customary to speak
about \emph{macroscopic behavior} for $\mu_{X_n}$ and about \emph{edge behavior}
for $\max_{1\leq k\leq n}|X_{n,k}|$.

\section{Proofs}\label{se:proofs}

\subsection{Proof of Lemma \ref{le:conf}}
\label{se:le:conf}

\begin{proof}[Proof of Lemma \ref{le:conf}]
	Following for instance \cite[Ch. 0, Example 5.7 ]{MR1485778}, we have
	\begin{multline}\label{eq:UR}
	U_{\rho}(x) =\frac{1}{\pi R^2}
	\int_0^{2\pi}\int_0^R\log\frac{1}{|x-r\mathrm{e}^{\mathrm{i}\theta}|}r\mathrm{d}r\mathrm{d}\theta\\
	=\frac{1}{2}\Bigr(1-\frac{|x|^2}{R^2}\Bigr)\mathbf{1}_{|x|\leq
		R}-\log\frac{|x|}{R}\mathbf{1}_{|x|>R}-\log R,
	\end{multline}
	which is harmonic outside the support of $\rho$, in
	accordance with \eqref{eq:inv}.
	Now we define 
	\[
	{W_\rho=-U_\rho}
	\quad\text{and}\quad
	G(x,y) = g(x-y) + W_\rho(x) + W_\rho(y)
	\]
	then
	\[\sum_{i<j} G(x_i,x_j)
	=\sum_{i<j}g(x_i-x_j) + 
	(n-1) \sum_{i=1}^n W_\rho(x_i)
	\]
	and
	\begin{align*}
	\mathrm{e}^{-\beta E_n(x_1,\ldots,x_n)}
	&=\mathrm{e}^{-\beta \left( 
		\sum_{i<j}g(x_i-x_j) + \alpha 
		\sum_{i=1}^n W_\rho (x_i) 
		\right)}\\
	&=
	\mathrm{e}^{-\beta \sum_{i<j}G(x_i,x_j)}
	\prod_{i=1}^n
	\mathrm{e}^{-\beta \left(\alpha - n+1 \right)W_\rho(x_i)}.
	\end{align*}
	The idea now is to show that the first exponential in the last display is
	bounded whereas the product of exponentials is integrable. Indeed, we shall
	use the following properties:
	\begin{enumerate}[(a)]
		\itemsep0.1em 
		\item \label{en:Confining} The function
		$x\mapsto\left|W_\rho(x) -\log|x|\right|$ is bounded for $|x| \geq 1$;
		\item \label{en:BoundedBelow} The function $G$ is bounded from below (see
		Remark \ref{rk:conf});
		\item \label{en:BoundedFar} 
		For all closed set $F$ and every
		compact set $K$ such that $F \cap K = \varnothing$, we have
		\[
		\sup_{(x,y) \in F \times K}|G(x,y)| < \infty.
		\]
	\end{enumerate}
	By \eqref{en:Confining} and since $W_\rho$ is bounded from
	below, we have
	\[
	x\in\mathbb{C}\mapsto\mathrm{e}^{- \beta \left(\alpha - n+1 \right)
		W_\rho(x)} \text{ is integrable }
	\Longleftrightarrow \beta(\alpha - n +
	1) > 2.
	\]
	Thus, using additionally \eqref{en:BoundedBelow}, we obtain
	\[
	\beta(\alpha - n + 1) > 2 \Longrightarrow Z_n < \infty.
	\]
	For the converse implication choose $n-1$ pairwise disjoint compact
	sets $K_1,\dots,K_{n-1}$ in a {neighborhood of the origin, say
		the open unit disc.}  Then, \eqref{en:BoundedFar} implies that
	$ -\beta \sum_{i<j} G(x_i,x_j) $ is bounded from below whenever
	$(x_1, \dots,x_{n-1}) \in K_1 \times \dots \times K_{n-1}$ and
	$|x_n| \geq 1$. As $W_\rho$ is bounded from above in the unit disc
	there exists a constant $C$ such that the integrand is bounded from
	below by
	\[
      \frac{C}{|x_n|^{\beta(\alpha - n + 1)}}
    \]
	whenever $(x_1,\dots,x_{n-1}) \in 
	K_1 \times \dots \times K_{n-1}$ and
	$|x_n| \geq 1$. We conclude that
	\[
      \beta(\alpha - n + 1) \leq 2\Longrightarrow Z_n = \infty.
    \]  
\end{proof}

\begin{remark}[Confinement or integrability condition]\label{rk:conf}
	The integrability condition in Lemma \ref{le:conf} can also be
	derived using the elementary inequality $|a-b|\leq(1+|a|)(1+|b|)$
	valid for all $a,b\in\mathbb{C}$, as in \cite{MR3215627}. Namely, it
	gives $\prod_{i<j}|x_i-x_j|\leq(\prod_{i=1}^n(1+|x_i|))^{n-1}$ since
	each $i$ appears exactly in $n-1$ elements of
	$\{(i,j):i<j\}$. Hence, for all $x_1,\ldots,x_n\in\mathbb{C}$ such
	that $|x_1|>R,\ldots,|x_n|>R$, we have, for some constants $c',c''$,
	\[
	E_n(x_1,\ldots,x_n)
	\geq c'-\sum_{i<j}\log|x_i-x_j|+\alpha\sum_{i=1}^n\log|x_i|
	\geq c''-(\alpha-(n-1))\sum_{i=1}^n\log|x_i|.
	\]
	Therefore $Z_n<\infty$ if $x\mapsto 1/|x|^{\beta(\alpha-(n-1))}$ is
	integrable at infinity with respect to the Lebesgue measure on
	$\mathbb{C}$, which holds {when}
	$\beta(\alpha-(n-1))>2$.
\end{remark}

\subsection{Proof of Theorem \ref{th:low}}
\label{se:th:low}

We use the method used in \cite{david-ihp} for the proof of \eqref{eq:LDP}.
This method takes its roots in \cite{dupuis2015large}. In this approach, the
relative entropy plays an essential role for tightness, even when it does not
appear in the final statement. Beware also that we have here to take into account
the fact that the potential $V=-\frac{\alpha}{n} U_\rho$ depends on $n$.

\begin{proof}[Proof of Theorem \ref{th:low}]
	Suppose first that $\lambda > 1$. Define
	$G: \mathbb C \times \mathbb C \to (-\infty,\infty]$ by
	\[
	G(x,y) = g(x-y) + W_\rho(x) + W_\rho(y)
	\]
	and
	\[
	A_n = \frac{\alpha_n - n + 1}{n} - \frac{4}{n \beta_n}.
	\]
	Then,
	\[
	\lim_{n \to \infty} A_n = \lambda - 1 > 0
	\]
	and $\mathrm{e}^{-\beta_nE_n(x_1,\ldots,x_n)}$ writes
	\begin{align*}
	\exp &\left[-\beta_n \left( 
	\sum_{i<j}g(x_i-x_j) + \alpha_n 
	\sum_{i=1}^n W_\rho (x_i) 
	\right) \right]				\\
	&=
	\exp \left[-\beta_n 
	\left(\sum_{i<j}G(x_i,x_j)
	+ A_n n \sum_{i=1}^n W_\rho(x_i) \right)
	\right]
	\prod_{i=1}^n
	\exp 
	\left[ - 4
	W_\rho(x_i) \right],
	\end{align*}
	where $W_\rho=-U_\rho$. Let us define
	$H_n: \mathbb C^n \to (-\infty,+\infty]$ by
	\[
	H_n(x_1,\dots,x_n)%
	=\frac{1}{n^2}\sum_{i<j}G(x_i,x_j)+A_n \frac{1}{n} \sum_{i=1}^n W_\rho(x_i) 
	\]
	and 
	\[
	\mathrm{d}\sigma(x)=\frac{\mathrm{e}^{-4W_\rho}}{Z_\sigma}\mathrm{d}\ell_{\mathbb{C}}(x)
	\quad\text{where}\quad
	Z_\sigma =\int_{\mathbb C}\mathrm{e}^{-4W_\rho(x)}\mathrm{d}\ell_{\mathbb{C}}(x),
	\]
	so that the Coulomb gas law
	$\mathrm{e}^{-\beta_nE_n}\mathrm{d}\ell_{\mathbb{C}}(x)$ is proportional to
	\[
	\mathrm{e}^{-n^2\beta_n H_n} \mathrm{d}\sigma^{\otimes_n}(x_1,\dots,x_n).
	\]
	Take any bounded continuous $f:\mathbb C \to \mathbb R$. Then following for
	instance \cite{david-ihp}, we get that
	\begin{multline*}
	\frac{1}{n^2\beta_n} 
	\log \int_{\mathbb C^n}\mathrm{e}^{-n^2\beta_n 
		\left(f \circ i_n + H_n \right)} 
	\mathrm{d}\sigma^{\otimes_n}(x_1,\dots,x_n)\\
	=
	-\inf_{\mu \in \mathcal{P}(\mathbb C^n )}
	\Bigr\{ 
	\mathbb E_{\mu} [f \circ i_n + H_n]
	+ \frac{1}{n^2\beta_n}D(\mu \mid \sigma^{\otimes_n} )
	\Bigr\},
	\end{multline*}
	where $\mathcal{P}(\mathbb{C}^n)$ is the set of probability measures on
	$\mathbb{C}^n$ and where
	\[
	i_n(x_1,\ldots,x_n)=\frac{1}{n}\sum_{i=1}^n\delta_{x_i},
	\]
	so that it is natural to expect that
	\[
	-\inf_{\mu \in \mathcal{P}(\mathbb C^n )} \Bigr\{ \mathbb E_{\mu} [f \circ
	i_n + H_n] + \frac{1}{n^2 \beta_n}D(\mu \mid \sigma )\Bigr\}
	\]
	converges to
	\[
	-\inf_{\mu \in \mathcal{P}(\mathbb C)}
	\Bigr\{ f(\mu) + 
	\frac{1}{2}
	\int_{\mathbb C \times \mathbb C}
	G(x,y) \mathrm{d}\mu(x) \mathrm{d}\mu(y)
	+ (\lambda - 1) 
	\int_{\mathbb C} W_\rho(x) \mathrm{d} \sigma(x)\Bigr\}
	\]
	which is what exactly happens. We refer to \cite{david-ihp} for the details.
	
	Suppose now that $\lambda = 1$. Define
	$H_n: \mathbb C^n \to (-\infty,\infty]$ by
	\[
	H_n(x_1,\dots,x_n)
	=
	\frac{1}{n^2}\sum_{i<j}G(x_i,x_j)
	\]
	where
	\[
	G(x,y) = g(x-y) + W_\rho(x) + W_\rho(y)
	\]
	and define
	\[
	\gamma_n= \beta_n(\alpha_n - n + 1).
	\]
	and
	\[
	\mathrm{d} \sigma_n(x)
	=\frac{\mathrm{e}^{-\gamma_n W_\rho}}{Z_{\sigma_n}}
	\mathrm{d}\ell_{\mathbb{C}}(x).
	\]
	The Coulomb gas law is proportional to
	\[
	\mathrm{e}^{-n^2\beta_n H_n}
	\mathrm{d} \sigma_n^{\otimes_n} (x_1,\dots,x_n).\]
	As before, we notice
	that
	\begin{multline*}
	\frac{1}{n^2\beta_n} 
	\log\int_{\mathbb C^n}\mathrm{e}^{-n^2\beta_n 
		\left(f \circ i_n + H_n \right)} 
	\mathrm{d}\sigma_n^{\otimes_n}(x_1,\dots,x_n)
	\\= 
	-\inf_{\mu \in \mathcal{P}(\mathbb C^n )}
	\Bigr\{ 
	\mathbb E_{\mu} [f \circ i_n + H_n]
	+ \frac{1}{n^2\beta_n}D(\mu \mid
	\sigma_n^{\otimes_n} )
	\Bigr\},
	\end{multline*}
	but now we remark that
	if $\mathrm{d} \mu = \rho \, \mathrm{d} \ell_{\mathbb C}$
	we have
	\[
	\frac{1}{n\beta_n}D(\mu \mid
	\sigma_n^{\otimes_n} )
	=
	\frac{1}{n \beta_n} \int_{\mathbb C} \rho \log \rho \, \mathrm{d} \ell_{\mathbb C}
	+ \frac{\gamma_n}{n \beta_n} \int_{\mathbb C}
	W_\rho \mathrm{d} \mu
	+ \frac{\gamma_n}{n \beta_n}
	\frac{1}{\gamma_n} \log \int_{\mathbb C}
	\mathrm{e}^{-\gamma_n W_\rho} \mathrm{d} \ell_{\mathbb C}
	\]
	if all these terms make sense (holds if $\rho$ is bounded and compactly
	supported) and, thus,
	\[
	\lim_{n \to \infty}
	\frac{1}{n\beta_n}D(\mu \mid
	\sigma_n^{\otimes_n} ) = 0.
	\]
	By the same arguments as before we may conclude.
\end{proof}

\subsection{Proof of Theorem \ref{th:high}}
\label{se:th:high}

As in the proof of Theorem \ref{th:low}, we proceed as in \cite{david-ihp} for the proof of (\ref{eq:LDP}),
taking into account the fact that the potential $V=-\alpha U_\rho$ depends on $n$.

\begin{proof}[Proof of Theorem \ref{th:high}]
	Choose any $\delta > 0$ such that $2 < \delta < \beta( \lambda - 1)$ and
	define
	\[
	A_n = \frac{\alpha_n - n + 1}{n} - \frac{\delta}{n\beta_n}.
	\]
	Then,
	\[
	\lim_{n\to\infty}A_n = \lambda - 1 - \frac{\delta}{\kappa} > 0
	\]
	and we write
	\begin{align*}
	\exp &\left[-\beta_n \left( 
	\sum_{i<j}g(x_i-x_j) + \alpha_n 
	\sum_{i=1}^n W_\rho (x_i) 
	\right) \right]				\\
	&=
	\exp \left[-\beta_n 
	\left(\sum_{i<j}G(x_i,x_j)
	+ A_n n \sum_{i=1}^n W_\rho(x_i) \right)
	\right]
	\prod_{i=1}^n
	\exp 
	\left[ - \delta
	W_\rho(x_i) \right]
	\end{align*}
	where $W_\rho=-U_\rho$ and $G(x,y)=g(x-y)+W_\rho(x)+W_\rho(y)$ are as in the
	proof of Theorem \ref{th:low}. Furthermore, by following the same ideas as
	for the case $\lambda > 1$ in the proof of Theorem \ref{th:low}, we conclude
	that $\mu_{X_n} \to \mu_*$ as $n\to\infty$, where $\mu_*$ is the unique
	minimizer of
	\[
	\mu \mapsto 
	\frac{\kappa}{2}\int_{\mathbb C \times \mathbb C} G(x,y) \mathrm{d}\mu(x)
	\mathrm{d}\mu(y) + \left[ \kappa(\lambda -1 ) - \delta \right]\int_{\mathbb
		C} W_\rho \mathrm{d} \mu + D\Bigr(\mu\mid\frac{\mathrm{e}^{-\delta W_\rho}}{Z_\sigma}\mathrm{d} \ell_{\mathbb C}\Bigr).
	\]
	We refer to \cite{david-ihp} for the details. Then, still following
	\cite{david-ihp}, we can get from \eqref{eq:mustar} that
	\[
	\Delta \log\varphi= 2\pi\kappa\varphi -
	2\pi\kappa\lambda\frac{\mathbf{1}_{\left|\cdot\right|\leq R}}{\pi R^2}.
	\]
\end{proof}

\subsection{Proof of Theorem \ref{th:edge}}
\label{se:th:edge}

\begin{proof}[Proof of Theorem \ref{th:edge}]
  It is enough to prove the fluctuation result since it implies the
	convergence in probability to the edge (just like the central limit
	theorem implies the weak law of large numbers). Now, since $\beta=2$,
  $P_n$ is determinantal and we can use \eqref{eq:kostlan} which we rephrase
  as follows: the point process of the radii $\{|X_{n,1}|,\dots,|X_{n,n}|\}$
  has the same law as the point process $\{|Y_{n,0}|,\dots,|Y_{n,n-1}|\}$
  where $Y_{n,0},\dots,Y_{n,n-1}$ are independent (not identically
  distributed) complex random variables such that
  \[
	Y_{n,k} \sim a_{n,k} |z|^{2k} \mathrm{e}^{-2(n+\kappa_n) W_\rho(z)} %
	\mathrm d \ell_{\mathbb C}(z) %
	\quad\text{with}\quad %
	a_{n,k} = \left( \int_{\mathbb C}
      |z|^{2k}\mathrm{e}^{-2(n+\kappa_n)W_\rho(z)} %
      \mathrm d \ell_{\mathbb C} (z) \right)^{-1}
  \]
  is a normalization constant, and $W_\rho=-U_\rho$. In particular,
  \begin{equation}\label{eq:KostlanMax}
	\mathbb{P}(\max_{1\leq k\leq n}|X_{n,k}|\leq x)
	= \prod_{k=0}^{n-1}\left(%
      1 - a_{n,k}\int_{D_x^c} |z|^{2k}\mathrm{e}^{-2(n+\kappa_n) W_\rho(z)} 
      \mathrm d \ell_{\mathbb C}(z)\right).
  \end{equation}
  By adding a constant, we shall suppose that $W_\rho(z) = \log R$ if
  $|z| = R$. In that case
  \begin{equation}\label{eq:Vlog}
	W_\rho(z) =\log|z|\quad\text{if}\quad |z| \geq R%
	\quad\text{while}\quad%
	W_\rho(z) > \log|z|\quad\text{if}\quad |z| < R.
  \end{equation}
  Suppose that $x > R$. Then, by \eqref{eq:KostlanMax} and \eqref{eq:Vlog},
  \[
	\mathbb P(\max_{1\leq k\leq n}|X_{n,k}|\leq x) 
	= \prod_{k=0}^{n-1}\Bigr(1 - a_{n,k}\int_{D_x^c} |z|^{2(k-n-\kappa_n) }
	\mathrm d \ell_{\mathbb C}(z)
	\Bigr).
  \]
  Using the change of indices $k \to n-k-1$
  we obtain
  \[
	\mathbb P(\max_{1\leq k\leq n}|X_{n,k}| \leq x) 
	= \prod_{k=0}^{n-1}
	\Bigr(1 - a_{n,n-k-1}\int_{D_x^c} |z|^{-2(k+\kappa_n+1) }
	\mathrm d \ell_{\mathbb C}(z)
	\Bigr).
  \]
  The limit can be calculated by using Lebesgue's dominated convergence
  theorem.
  
  \emph{Domination}. First we observe the domination
  \begin{equation}\label{eq:Domination}
	1 - b_{n,n-k-1}\int_{D_x^c} |z|^{-2(k+\kappa_n+1) }
	\mathrm d \ell_{\mathbb C}(z)
	\leq
	1 - a_{n,n-k-1}\int_{D_x^c} |z|^{-2(k+\kappa_n+1) }
	\mathrm d \ell_{\mathbb C}(z)
  \end{equation}
  where
  \[
	b_{n,n-k-1}= \Bigr(\int_{D_R^{c}} |z|^{-2(k+\kappa_n+1)}
	\mathrm{d}\ell_{\mathbb C} (z)\Bigr)^{-1}.
  \]
  But the left-hand side of \eqref{eq:Domination} can be calculated explicitly
  as
  \[
	1 - b_{n,n-k-1}\int_{D_x^c} |z|^{-2(k+\kappa_n+1) }
	\mathrm{d}\ell_{\mathbb C}(z) %
	=1 - \Bigr(\frac{R}{x} \Bigr)^{2(k + \kappa_n)}
  \]
  and, if we choose $\varepsilon \in (0, \kappa)$, then, for $n$ large enough
  we have $\kappa - \varepsilon <\kappa_n$ so that
  \[
	1 - \Bigr(\frac{R}{x} \Bigr)^{2(k + \kappa - \varepsilon)} \leq 1 -
	\Bigr(\frac{R}{x} \Bigr)^{2(k + \kappa_n)}.
  \]
  Since
  \[
	\prod_{k=0}^\infty\Bigr( 1 - \Bigr(\frac{R}{x} \Bigr)^{2(k + \kappa -
      \varepsilon)} \Bigr)>0
  \]
  we have a domination from below of our product.
  
  \emph{Pointwise convergence.} Now, let us see the convergence of the terms.
  The coefficient
  \begin{align*}
	a_{n,n-k-1}^{-1}
	& =\int_{\mathbb C} |z|^{2(n-k-1)}\mathrm{e}^{-2(n +\kappa_n)W_\rho(z)}
   \mathrm d \ell_{\mathbb C} (z)   				\\
	&=\int_{\mathbb C} |z|^{-2(k+1)}\mathrm{e}^{-2\kappa_n W_\rho(z)}
   \mathrm{e}^{-2n 
   \left( W_\rho(z) - \log|z| \right)}
   \mathrm d \ell_{\mathbb C} (z)   
  \end{align*}
  has an integrand which converges to
  \[
	|z|^{-2(k+1)}\mathrm{e}^{-2\kappa W_\rho(z)}\mathbf{1}_{|z|\geq R}
	= |z|^{-2(k+\kappa + 1)}\mathbf{1}_{|z|\geq R}
  \]
  and that is dominated by, for instance,
  $|z|^{-2(k+1)}\mathrm{e}^{-(\kappa+\varepsilon)W_\rho(0)}
  \mathrm{e}^{-2(\kappa - \varepsilon)W_\rho(z)}$ which is integrable. So,
  \[
	\lim_{n \to \infty}a_{n,n-k-1}^{-1}
	= \int_{D_R^c}|z|^{-2(k+\kappa + 1)}
	\mathrm d \ell_{\mathbb C}(z)
  \]
  and, by evaluating the integrals,
  \[
	\lim_{n \to \infty}
	\Bigr(1 - a_{n,n-k-1}^{-1}\int_{D_x^c} |z|^{-2(k+\kappa_n+1) }
	\mathrm d \ell_{\mathbb C}(z) \Bigr)
	= 1 - \Bigr(\frac{R}{x} \Bigr)^{2(k+\kappa)}.
  \]
  
  \emph{Lebesgue's dominated convergence theorem.} Having dominated each term
  from below and having proved the convergence of each term we apply
  Lebesgue's dominated convergence theorem and obtain the desired result.
\end{proof}

\subsection{Proof of Theorem \ref{th:gumbel}}
\label{se:th:gumbel}

\begin{proof}[Proof of Theorem \ref{th:gumbel}]
  It is enough to prove the fluctuation result since it implies the
  convergence in probability to the edge. Now, let us write
  $\alpha_n=n\lambda_n$ with $\lambda_n\to\lambda>1$. From the formula for
  $V=-\frac{\alpha_n}{n}U_\rho$ of Lemma \ref{le:conf}, we get that
  $V=V^{\mathrm{Gin}}$ on $D_R$, where
  \[
	V^{\mathrm{Gin}}=\frac{\lambda_n}{2R^2}\left|\cdot\right|^2.
  \]
  Let $P_n^{\mathrm{Gin}}$ be the Boltzmann--Gibbs probability measure on
  $\mathbb{C}^n$ defined by \eqref{eq:P} with potential $V^{\mathrm{Gin}}$,
  which is a (scaled) complex Ginibre ensemble. It follows that for any event
  \[
	A\subset D_R^n%
	=D_R\times\cdots\times D_R%
	=\{x\in\mathbb{C}^n:\max(|x_1|,\ldots,|x_n|)\leq R\}
  \]
  we have
  \begin{equation}\label{eq:gequiv}
	P_n(A)=\frac{Z_n^{\mathrm{Gin}}}{Z_n}P_n^{\mathrm{Gin}}(A).
  \end{equation}
  
  From Theorem \ref{th:high}, the limiting distribution $\mu_*$ under $P_n$ is
  supported in $D_{R/\sqrt{\lambda}}$. Now $\lambda>1$ implies
  $D_{R/\sqrt{\lambda}}\subsetneq D_R$. For an arbitrary
  $\varepsilon\in(0,R(1-1/\sqrt{\lambda})]$, let us define the event
  \[
	A_n=D_{R/\sqrt{\lambda}+\varepsilon}^n
	=\Bigr\{x\in\mathbb{C}^n:
	\max(|x_1|,\ldots,|x_n|)\leq \frac{R}{\sqrt{\lambda}}+\varepsilon\Bigr\}
	\subset D_R^n.
  \]
  
  Let $a_n,b_n$ and $G$ be as in Theorem \ref{th:gumbel} and let $\xi\sim G$.
  Let us define
  \[
	M_n=a_n(\max(|x_1|,\ldots,|x_n|)-b_n).
  \]
  Then it is known, see \cite{MR1986426,MR3215627}, that
  \begin{equation}\label{eq:gedge}
	\lim_{n\to\infty}P_n^{\mathrm{Gin}}(A_n^c)=0
	\quad\text{and}\quad
	\lim_{n\to\infty}\mathbb{E}_{P_n^{\mathrm{Gin}}}(\mathrm{e}^{\mathrm{i}\theta M_n})
	=\mathbb{E}(\mathrm{e}^{\mathrm{i}\theta\xi}),
	\quad\theta\in\mathbb{R}.
  \end{equation}
  
  \begin{lemma}[Partition functions]\label{le:ZZ}
    \begin{equation}\label{eq:ZZ}
      \lim_{n\to\infty}\frac{Z_n^{\mathrm{Gin}}}{Z_n}=1.
    \end{equation}
  \end{lemma}
  
  \begin{proof}[Proof of Lemma \eqref{le:ZZ}]
    Since $Z_nP_n(A_n)=Z_n^{\mathrm{Gin}}P_n^{\mathrm{Gin}}(A_n)$ and since
    $\lim_{n\to\infty}P_n^{\mathrm{Gin}}(A_n)=1$ from \eqref{eq:gedge}, the
    desired statement is actually equivalent to
    \begin{equation}\label{eq:jedge}
      \lim_{n\to\infty}P_n(A_n)=1.
    \end{equation}
    But from \eqref{eq:kostlan}, we obtain, denoting $V=Q(\left|\cdot\right|)$
    and $r=R/\sqrt{\lambda}+\varepsilon$,
    \[
      P_n(A_n)=\prod_{k=1}^n\Bigr(1-c_{n,k}\int_r^\infty
      t^{2k-1}\mathrm{e}^{-2nQ(t)}\mathrm{d}t\Bigr)
      \quad\text{where}\quad
      c_{n,k}^{-1}=\int_0^\infty t^{2k-1}\mathrm{e}^{-2nQ(t)}\mathrm{d}t.
    \]
    It is possible to follow this elementary route and to evaluate the limit
    of this product by evaluating the integrals. Actually \eqref{eq:jedge} is
    a weak consequence of \cite{ameur2019localization}, which is itself a
    refinement of \cite[Theorem 1.12]{chafai2018concentration}. Note that
    \cite[Theorem 1]{ameur2019localization} deals with potentials not
    necessarily rotationally invariant, and provides a quantitative exponential
    upper bound on the probability. Note also that condition (vi) of
    \cite{ameur2019localization} translates in our context to
    $\int_\mathbb{C} U_\rho\mathrm{e}^{2\lambda
      U_\rho}\mathrm{d}\ell_{\mathbb{C}}<\infty$ which holds precisely when
    $\lambda>1$. When $\lambda=1$, Theorem \ref{th:edge} shows that
    \eqref{eq:jedge} does not hold, so that $\lambda>1$ is the optimal
    condition on $\lambda$ under which \eqref{eq:jedge} can hold.
  \end{proof}
  
  Now, using \eqref{eq:gequiv}, \eqref{eq:ZZ}, and \eqref{eq:gedge}, we
  obtain
  \[
	|\mathbb{E}_{P_n}(\mathrm{e}^{\mathrm{i}\theta M_n})
	-\mathbb{E}_{P_n}(\mathrm{e}^{\mathrm{i}\theta M_n}\mathbf{1}_{A_n})|
	=|\mathbb{E}_{P_n}(\mathrm{e}^{\mathrm{i}\theta M_n}\mathbf{1}_{A_n^c})|
	\leq P_n(A_n^c)
	=\frac{Z_n^{\mathrm{Gin}}}{Z_n}P_n^{\mathrm{Gin}}(A_n^c)
	\underset{n\to\infty}{\longrightarrow}0.
  \]
  Next, and similarly, using \eqref{eq:gequiv},
  \[
	\mathbb{E}_{P_n}(\mathrm{e}^{\mathrm{i}\theta M_n}\mathbf{1}_{A_n})
	=
	\frac{Z_n^{\mathrm{Gin}}}{Z_n}
	\mathbb{E}_{P_n^{\mathrm{Gin}}}(\mathrm{e}^{\mathrm{i}\theta M_n}\mathbf{1}_{A_n})
	=
	\frac{Z_n^{\mathrm{Gin}}}{Z_n}
	\mathbb{E}_{P_n^{\mathrm{Gin}}}(\mathrm{e}^{\mathrm{i}\theta M_n})
	-
	\frac{Z_n^{\mathrm{Gin}}}{Z_n}\mathbb{E}_{P_n^{\mathrm{Gin}}}(\mathrm{e}^{\mathrm{i}\theta M_n}\mathbf{1}_{A_n^c}),
  \]
  and using \eqref{eq:ZZ} and \eqref{eq:gedge} we obtain
  \[
	\lim_{n\to\infty}
	|\mathbb{E}_{P_n}(\mathrm{e}^{\mathrm{i}\theta M_n}\mathbf{1}_{A_n})    
	-\mathbb{E}(\mathrm{e}^{\mathrm{i}\theta\xi})|
	=0.    
  \]
  Finally we have obtained as expected that 
  \[
	\lim_{n\to\infty}\mathbb{E}_{P_n}(\mathrm{e}^{\mathrm{i}\theta M_n})
	=\lim_{n\to\infty}\mathbb{E}_{P_n^{\mathrm{Gin}}}(\mathrm{e}^{\mathrm{i}\theta M_n})
	=\mathbb{E}(\mathrm{e}^{\mathrm{i}\theta\xi}),\quad\theta\in\mathbb{R}.
  \]
\end{proof}

\begin{remark}
	An alternative and purely equivalent way to formulate the proof of
	Theorem \ref{th:gumbel} is to note first that \eqref{eq:gequiv} is
	indeed equivalent to stating that
	\begin{equation}\label{eq:gequivcond}
	P_n(\cdot\mid D_R^n)=P_n^{\mathrm{Gin}}(\cdot\mid D_R^n).
	\end{equation}
	Next, we may deduce from \eqref{eq:gequiv}, \eqref{eq:ZZ}, and
	\eqref{eq:gedge} that
	\[
	\lim_{n\to\infty}P_n(D_R^n)=\lim_{n\to\infty}P_n^{\mathrm{Gin}}(D_R^n)=1.
	\]
	which implies using \eqref{eq:gequivcond} and Lebesgue's dominated
	convergence theorem that
	\[
	\lim_{n\to\infty}
	\Bigr(P_n(\cdot\mid D_R^n)-P_n^{\mathrm{Gin}}(\cdot\mid D_R^n)\Bigr)=0
	\]
	weakly in $\mathcal{P}(\mathbb{C})$, which implies in turn using again
	\eqref{eq:gedge} that
	\[
	\lim_{n\to\infty}\mathbb{E}_{P_n}(\mathrm{e}^{\mathrm{i}\theta M_n})
	=\lim_{n\to\infty}\mathbb{E}_{P_n^{\mathrm{Gin}}}(\mathrm{e}^{\mathrm{i}\theta
		M_n})
	=\mathbb{E}(\mathrm{e}^{\mathrm{i}\theta\xi}),\quad\theta\in\mathbb{R}.
	\]
\end{remark}

\section{Historical comments}\label{se:hist}

The \emph{jellium} model was used around 1938 by Eugene P.\ Wigner in
\cite{TF9383400678} for the modeling of electrons in metals, more than ten
years before his renowned works on random matrices. This model was inspired
from the Hartree--Fock model of quantum mechanics, see
\cite{guiuliani,Lieb1975,MR2583992,MR3309890}, and
\cite{lewin-lieb-seiringer,MR3732693}. The term \emph{jellium} was apparently
coined by Conyers Herring since the smeared charge could be viewed as a
positive ``jelly'', see \cite{hughes2006theoretical}. The model is also known
as a \emph{one-component plasma with background}. As already mentioned,
usually charge-neutral jellium models are studied, and this is done typically
after restricting the electrons to live on some compact support of positive
background. The restriction ensures integrability of the energy and the
interest is usually focused on the distribution/behavior of electrons in the
``bulk'' of the limiting system when the volume of the compact set goes to
infinity (thermodynamic limit). There are some exceptions where the edge has
been considered, for instance in \cite{can2015exact}. Also, the edge of
Laughlin states has been considered in \cite{can2014singular,
  gromov2016boundary}.

The case $d=3$ is considered in \cite{Lieb1975}, and
quoting \cite{Lieb1975}: ``\emph{It is also possible to consider the one- and
	two-dimensional versions of this problem, where the Coulomb potential
	$|x|^{-1}$ is replaced by $-|x|$ and $-\log|x|$, respectively. In the
	one-dimensional, classical case, Baxter \cite{baxter_1963} calculated the
	partition function exactly. For that case, Kunz \cite{KUNZ1974303} showed
	that the one-particle distribution function exists and that it has
	crystalline ordering, i.e., the Wigner lattice exists for all temperatures.
	Brascamp and Lieb \cite{Brascamp2002} showed the same to be true in the
	quantum mechanical case for one-component fermions when $\beta$ is large
	enough. Although we do not deal with the one-dimensional problem here, our
	methods would apply in that case. In two dimensions there are difficulties
	connected with the long-range nature of the $-\log|x|$ potential, and we
	shall not discuss this here.}'' For more background literature on the
jellium, see also \cite{refId0,
	Jancovici1993,FORRESTER1998235,aizenman2010symmetry,jansen2014wigner,MR3664397}.
See in particular \cite{levesque2000charge}, for the fluctuations of
non-neutral jelliums.

Historically, Coulomb gas models appeared naturally in statistics around
1920-1930 in the study of the spectrum of empirical covariance matrices of
Gaussian samples. Nowadays we speak about the Laguerre ensemble and Wishart
random matrices. This was almost ten years before the introduction of the
jellium model by Wigner. In the 1950's, Wigner rediscovered, by accident, these
works by reading a statistics textbook, and this motivated him to use random
matrices for the modeling of energy levels of heavy nuclei in atomic physics,
see \cite{blog}. We refer to \cite{MR2932622} for these historical aspects.
The work of Wigner was amazingly successful, and he received in 1963 a Nobel
prize in Physics ``\emph{for his contributions to the theory of the atomic
	nucleus and the elementary particles, particularly through the discovery and
	application of fundamental symmetry principles.}''. The term \emph{Coulomb
	gas} is explicitly used by Dyson in his first seminal 1962 paper
\cite{dyson1962statistical} and by Ginibre \cite{MR0173726}, while the term
\emph{Fermi-gas} was used earlier by Mehta \&\ Gaudin \cite{mehta1960density}
and also later by Dyson \&\ Mehta \cite{mehta1963statistical}.


\section*{Acknowledgments}

We would like to thank an anonymous reviewer for his useful remarks on the
presentation, and also Sabine Jansen, Mathieu Lewin, and Grégory Schehr for
their feedback.

\bibliographystyle{smfplain}
\bibliography{jellium2d}

\end{document}